\documentclass[12pt,leqno]{amsart}
\usepackage{amsmath,amsfonts,enumerate}
\usepackage{graphicx}
\usepackage{color}
\usepackage{ulem}

\numberwithin{equation}{section}

\newtheorem{theorem}{Theorem}[section]

\newtheorem{lemma}[theorem]{Lemma}

\theoremstyle{definition}

\newtheorem{remark}[theorem]{Remark}

\numberwithin{equation}{section}

\newcommand{\bi}{\begin{itemize}}
\newcommand{\ei}{\end{itemize}}
\newcommand{\ben}{\begin{enumerate}}
\newcommand{\een}{\end{enumerate}}

\newcommand\be{\begin{equation}}
\newcommand\ee{\end{equation}}
\newcommand\bea{\begin{eqnarray}}
\newcommand\eea{\end{eqnarray}}
\newcommand\beaa{\begin{eqnarray*}}
\newcommand\eeaa{\end{eqnarray*}}

\newcommand\bef{\begin{frame}}
\newcommand\enf{\end{frame}}

\renewcommand{\phi}{\varphi}

\newcommand{\pf}{\noindent{\bf Proof. }}

\newcommand{\baa}{\begin{array}}
\newcommand{\eaa}{\end{array}}

\newcommand\ld{{\lambda}}

\def\epsilon{\varepsilon}

\begin{document}

%
%
\title[MEMS with Robin BC]{Bifurcation diagram of a Robin boundary value problem arising in MEMS}

\author[J.-S. Guo]{Jong-Shenq Guo}
\address{Department of Mathematics, Tamkang University, Tamsui, New Taipei City 251301, Taiwan}
\email{jsguo@mail.tku.edu.tw}

\author[N.I. Kavallaris]{Nikos I. Kavallaris}
\address{Department of Mathematical and Physical Sciences, University of Chester, Thornton Science Park, Pool Lane, Ince, Chester CH2 4NU, UK}
\email{n.kavallaris@chester.ac.uk}

\author[C.-J. Wang]{Chi-Jen Wang}
\address{Department of Mathematics, National Chung Cheng University, Min-Hsiung, Chia-Yi 62102, Taiwan }
\email{cjwang@ccu.edu.tw}

\author[C.-Y. Yu]{Cherng-Yih Yu}
\address{Department of Mathematics, Tamkang University, Tamsui, New Taipei City 251301, Taiwan}
\email{cherngyi@mail.tku.edu.tw}

\thanks{Date: \today. Corresponding Author: C.-J. Wang}



\begin{abstract}
{We consider a parabolic problem with Robin boundary condition which arises when the edge of a micro-electro-mechanical-system (MEMS) device is connected with a flexible nonideal support.
Then via a rigorous analysis we investigate the structure of the solution set of the corresponding steady-state problem. We show that a critical value (the pull-in voltage) exists so that the system has
exactly two stationary solutions when the applied voltage is lower than this critical value,
one stationary solution for applying this critical voltage, and no stationary solution above
the critical voltage.

}
\end{abstract}

\maketitle
\setlength{\baselineskip}{18pt}

\section{Introduction}
\setcounter{equation}{0}

In this paper, we study the deformation of an elastic membrane inside a micro-electro mechanical system (MEMS).
We consider the case when the distance between the plate and the membrane is relative small compared to the length of the device.
In the case when we ignore the inertia and the device is embedded in an electric circuit with/without a capacitor,
the equation describing the operation of the MEMS is reduced to the following single parabolic equation
\be\label{mems}
u_t=\Delta u+ \frac{{\lambda}}{(1-u)^2[1+\alpha\int_\Omega(1-u)^{-1}dx]^2},\; x\in\Omega,\, t>0,
\ee
where $\ld$ is a positive constant which proportional to the square of the applied voltage.
{Besides $\alpha$ is a nonnegative parameter related to the ratio of the applied capacitances; $\alpha=0$ corresponds to a MEMS device without a capacitor,
whilst for $\alpha>0$ a capacitor is connected in series with the device.}
$\Omega$ is the domain of the plate and $u=u(x,t)$ denotes the displacement of the membrane towards the plate.
{It is also assumed that the gap between the elastic membrane and the rigid plate is small compared to the dimensions of the elastic membrane.}

Due to the support of MEMS device might be nonideal and flexible, 
we consider the edge of the membrane is connected with a flexible nonideal support so that
the following Robin boundary condition is imposed
\be\label{rbc}
\frac{\partial u}{\partial\nu}+\beta u=0\quad\mbox{on $\partial\Omega$},
\ee
where $\nu$ is the unit outer normal of $\partial\Omega$ and $\beta$ is a given positive constant.
For the derivation of this model, we refer the reader to \cite{DKN} (see also \cite{PT01,PB02,PD05,KS18}).

We are interested in the structure of stationary solutions of this MEMS problem.
We shall only focus on the case when the device is one-dimensional, namely the plate  is a rectangle and the deformation $u$ only depends on the horizontal direction.
Without loss of generality we may assume that $\beta=1$ and $\Omega=(-1,1)$.
Therefore, we study the following boundary value problem for $U(x)$:
\bea
&&-U''(x)=\frac{{\lambda}}{[1-U(x)]^2\{1+\alpha\int_{-1}^{1}[1-U(y)]^{-1}dy\}^2},\quad x\in(-1,1),\label{ode}\\
&&\pm U'(\pm 1)+U(\pm 1)=0.\label{bc}
\eea
Our main theorem on the bifurcation diagram of problem \eqref{ode}-\eqref{bc} in this paper reads

\begin{theorem}\label{th:main}
For each $\alpha\ge 0$, there exists a finite positive constant $\ld^*=\ld^*(\alpha)$ such that problem \eqref{ode}-\eqref{bc} has exactly two solutions when $\ld\in(0,\ld^*)$,
a unique solution when $\ld=\ld^*$, and no solution when $\ld>\ld^*$.
\end{theorem}

Theorem~\ref{th:main} includes both the local ($\alpha=0$) and nonlocal ($\alpha>0$) cases. It provides the existence of the supremum $\lambda^*$ of {the spectrum of}  problem \eqref{ode}-\eqref{bc}.
The existence of a finite $\lambda^*$ ({\it {pull-in} voltage} in MEMS terminology) is vital,
since the solution of problem {\eqref{mems}-\eqref{rbc}} supplemented with an initial condition for any nonnegative initial profile quenches for $\lambda>\lambda^*$ (cf. \cite{DKN}).
The case when the boundary of the membrane is clamped so that \eqref{rbc} is replaced by the zero Dirichlet boundary condition was
studied extensively. We refer the reader to,  e.g., \cite{L89,GGy07,EG08,GHW09,GK12,KLN16,KS18}.  See also \cite{GPW06,FMPS07,GzW08}.
Little work was done for the case with Robin boundary conditions (see, e.g., \cite{G91,DKN}).
In particular, in \cite{DKN}, the bifurcation diagram of problem \eqref{ode}-\eqref{bc} was given, but only numerically.
The main purpose of this work is to derive rigorously the bifurcation diagram of problem \eqref{ode}-\eqref{bc}.

In the sequel, we let $U=U(x;\lambda)$ be a (classical) solution of \eqref{ode}-\eqref{bc} for a given $\lambda>0$.
It is more convenient to use the new dependent variable $w(x):=1-U(x)$. Then $w$ satisfies
\bea
&&w''(x)=\frac{{\ld}}{w^2(x)\left[1+\alpha\int_{-1}^1 {w^{-1}(y)}dy\right]^{2}},\; -1<x<1,\label{w-eq}\\
&&\pm w'(\pm 1)=1-w(\pm 1).\label{wbc}
\eea
Since $u$ is the deformation of the membrane and $U$ is classical, we have $0<U<1$ in $[-1,1]$ and $U$ is strictly concave {due to the right hand of \eqref{ode} is positive.}
It also follows that $w$ is strictly convex and $0<w<1$ in $[-1,1]$.
In fact, both $U$ and $w$ are symmetric with respect to $x=0$ (see Lemma~\ref{sym} in \S2).

Set $a:=w(0)$  and $b:=w(1)$ for a solution $w$ of \eqref{w-eq}-\eqref{wbc}.
To study the bifurcation diagram, the first task is to derive some relations between $a$, $b$ and $\ld$.
Actually, this was done in \cite{DKN} (see, in particular, (2.8)-(2.10) in \cite{DKN}).
However, these only give some implicit relations between those three parameters.
It is not clear, in particular, whether $\ld$ is a function of another single parameter (either $a$ or $b$).
By introducing a new parameter, namely $s:=b/a$, we found that all parameters $a$, $b$ and $\ld$ are functions of this single variable $s$.
It turns out that $\ld$ is a function of $a$ (but not $b$), since $a$ is (and $b$ is not) a strictly monotone function of $s$  (see Figure~\ref{Fig1} and \S2 below).
With {this piece of information in hand}, we then derive the one-to-one correspondence between $\{(a,\ld(a))\mid a\in(0,1)\}$ and stationary solutions of \eqref{w-eq}-\eqref{wbc} in \S2.
Then our main result, Theorem~\ref{th:main}, is proved in \S3, by some delicate analysis of the derivative $\ld'(s)$.
In addition, by the proof of Theorem~\ref{th:main}, the value of $\ld^*$ is directly obtained.
In particular, we derive that   $\ld^*\approx 0.10871$ when $\alpha=0$ and $\ld^*\approx 2.38709$ when $\alpha=1$.

\begin{figure}
  \centering
\includegraphics[scale=0.55]{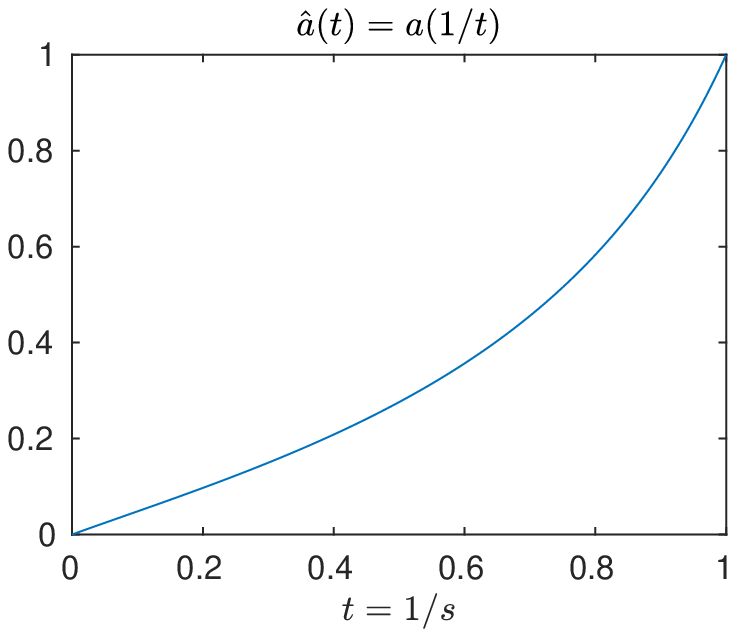}
\includegraphics[scale=0.55]{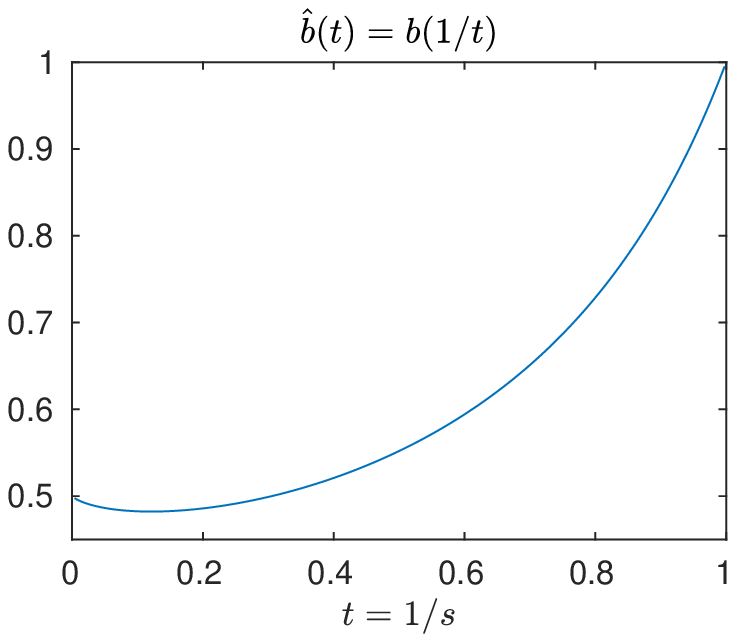}
\includegraphics[scale=0.55]{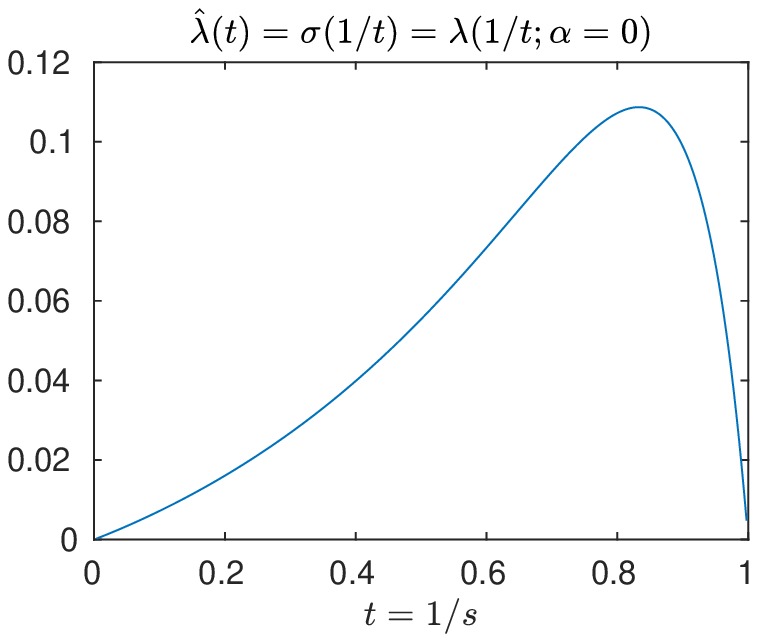}
  \caption{For the graph visualization, we set $t=1/s\in(0,1)$ as the horizontal axis, and plot the curves of {$\hat{a}(t)=a(1/t), \hat{b}(t)=b(1/t), \hat{\ld}(t)=\sigma(1/t)=\ld(1/t;\alpha=0)$ } derived from \eqref{as}-\eqref{n-lds}.}
  \label{Fig1}
\end{figure}


\section{Preliminaries}
\setcounter{equation}{0}

In the sequel, we let $w$ be a solution of \eqref{w-eq}-\eqref{wbc} and set
\be\label{sigma}
\sigma:=\ld\left[1+\alpha\int_{-1}^1 \frac{1}{w(y)}dy\right]^{-2}.
\ee
We first prove the following lemma.

\begin{lemma}\label{sym}
 Any solution $w$ of \eqref{w-eq}-\eqref{wbc} is symmetric with respect to $x=0$.
\end{lemma}

\begin{proof}
Since $0<w<1$, we have
\beaa
w'(-1)=w(-1)-1<0<1-w(1)=w'(1).
\eeaa
Hence there exists $\xi\in(-1,1)$ such that $w'(\xi)=0$.

Set $a:=w(\xi)\in(0,1)$. We first show that $w$ is symmetric with respect to $\xi$.

Multiplying \eqref{w-eq} by $w'$ and integrating  from $\xi$ to $x\in(\xi,1)$, we obtain
\be\label{wp}
\frac12(w')^2(x)=\sigma\left[\frac{1}{a}-\frac{1}{w(x)}\right],\; x\in(\xi,1].
\ee
Using $w'(x)>0$ for $x\in(\xi,1]$, it follows from \eqref{wp} that
\be\label{id1}
\frac{w'(x)}{\sqrt{1/a-1/w(x)}}=\sqrt{2\sigma},\;x\in(\xi,1].
\ee
Integrating \eqref{id1} from $\xi$ to $x\in(\xi,1]$ gives
\be\label{x-w}
\sqrt{a}\left\{\sqrt{w(x)[w(x)-a]}+a\ln\left(\frac{\sqrt{{w(x)}}+\sqrt{{w(x)-a}}}{\sqrt{a}}\right)\right\}=\sqrt{2\sigma}(x-\xi).
\ee
Similarly, for $x\in [-1,\xi)$, we have
\be\label{x-w1}
\sqrt{a}\left\{\sqrt{w(x)[w(x)-a]}+a\ln\left(\frac{\sqrt{{w(x)}}+\sqrt{{w(x)-a}}}{\sqrt{a}}\right)\right\}=\sqrt{2\sigma}(\xi-x).
\ee
Since the function $\sqrt{a}\{\sqrt{y(y-a)}+a\ln(\frac{\sqrt{{y}}+\sqrt{{y-a}}}{\sqrt{a}})\}$ is strictly increasing for $y\in (a,1)$, we deduce from \eqref{x-w} and \eqref{x-w1} that
$w(\xi+c)=w(\xi-c)$, and consequently $w'(\xi+c)=-w'(\xi-c)$ for all $c \in (0,\min\{1-\xi, 1+\xi\}]$.

Next, we show that $\xi=0$ by a contradiction argument. Without loss of generality we may assume that $\xi\in(-1,0)$.
Then the reflection point of $-1$ is $d=2\xi+1\in(\xi,1)$ and we have $w'(d)=-w'(-1)=1-w(-1)=1-w(d)>0$, due to $w\in(0,1)$.
Combining with the strict convexity of $w$, we get  that $w'(1)>w'(d)$ and $w(1)>w(d)$. This leads {to} a contradiction with $w'(1)=1-w(1)$.
Therefore, $\xi=0$ and hence any solution $w$ of \eqref{w-eq}-\eqref{wbc} is symmetric with respect to $x=0$.
The lemma is proved.
\end{proof}

Recall $a=w(0)$ and $b=w(1)$. We now derive some relations between $a$, $b$ and $\ld$.
Note that \eqref{wp}, \eqref{id1} and \eqref{x-w} hold with $\xi=0$, by Lemma~\ref{sym}.

First, using \eqref{wbc}, we obtain from \eqref{wp} that
\be\label{a-eq}
(1-b)^2=2\sigma\left(\frac{1}{a}-\frac{1}{b}\right).
\ee
Moreover, for $x=1$ (and $\xi=0$), we obtain from \eqref{x-w} that
\be\label{b-eq}
\sqrt{2\sigma}=\sqrt{a} {\left[\sqrt{b(b-a)}+a\ln\left(\frac{\sqrt{b}+\sqrt{b-a}}{\sqrt{a}}\right)\right]}.
\ee
Note that $0<a<b<1$. From \eqref{a-eq} and \eqref{b-eq} we deduce that
\be\label{ab}
\sqrt{b}(1-b)=\sqrt{b}(b-a)+a\sqrt{b-a}\,\ln{\left(\frac{\sqrt{b}+\sqrt{b-a}}{\sqrt{a}}\right)}.
\ee
This gives the relation between $a$ and $b$. Note that $\lambda=\sigma$ when $\alpha=0$.

When $\alpha>0$, we need to find $\ld$ in terms of $\sigma$, $a$ and $b$. Following \cite{DKN}, we first rewrite \eqref{id1} as
\be\label{id-n}
\sqrt{2\sigma}=\frac{1}{\sqrt{1/a-1/w(x)}}\frac{dw(x)}{dx},\;x\in(0,1].
\ee
Using \eqref{id-n} and \eqref{b-eq}, we compute
\bea
\int_{-1}^1\frac{1}{w(y)}dy&=&\frac{2}{\sqrt{2\sigma}}\int_a^b \frac{1}{w\sqrt{1/a-1/w}}dw=\frac{2\sqrt{a}}{\sqrt{2\sigma}}\int_a^b\frac{1}{\sqrt{w(w-a)}}dw\label{int}\\
&=&\frac{2}{\sqrt{b(b-a)}+a\ln(\frac{\sqrt{b}+\sqrt{b-a}}{\sqrt{a}})}\left\{2\ln\left(\frac{\sqrt{b}+\sqrt{b-a}}{\sqrt{a}}\right)\right\}:=I(a,b).\nonumber
\eea
Then we deduce from \eqref{sigma} that
\be\label{ld-ab}
\ld=\sigma[1+\alpha I(a,b)]^2.
\ee


Now, we introduce the variable $s:=b/a\in(1,\infty)$.
Then it follows from \eqref{ab} that
\be\label{as}
a=a(s)=\left(2s-1+\sqrt{\frac{s-1}{s}}A(s)\right)^{-1}.
\ee
Hereafter $A(s):=\ln(\sqrt{s}+\sqrt{s-1})$.
It is clear that $a(s)$ is strictly decreasing in $s$ such that $a(1^+)=1$ and $a(+\infty)=0$.
This also implies that $s$ is a function of $a$.
Using $b=sa$, we have
\be\label{bs}
b=b(s)=s\left(2s-1+\sqrt{\frac{s-1}{s}}A(s)\right)^{-1}.
\ee
With $s$ being the independent variable, we deduce from \eqref{b-eq} and \eqref{as} that
\be\label{lds}
\sigma=\sigma(s)=\frac{1}{2}\left[\sqrt{s}\sqrt{s-1}+A(s)\right]^{2}\left(2s-1+\sqrt{\frac{s-1}{s}}A(s)\right)^{-3}.
\ee
Moreover, using \eqref{int}, \eqref{ld-ab} and \eqref{lds}, we obtain that
\be\label{n-lds}
\ld=\ld(s)=\frac{\left[\sqrt{s(s-1)}+A(s)+4\alpha\left(2s-1+\sqrt{\frac{s-1}{s}}A(s)\right)A(s)\right]^2}{2\left[2s-1+\sqrt{\frac{s-1}{s}}A(s)\right]^3}.
\ee

Conversely, given an $s\in(1,\infty)$.  Let $a=a(s)$, $\sigma=\sigma(s)$ and $\ld=\ld(s)$ be defined by \eqref{as}, \eqref{lds} and \eqref{n-lds}, respectively.
Then, with these $a$, $\sigma$ and $\lambda$, the function $w(x)$ determined uniquely by \eqref{x-w} (with $\xi=0$) for each $x\in(0,1]$ is the solution of \eqref{w-eq} with initial condition $w(0)=a$ and $w'(0)=0$
such that $w(1)=b=1-w'(1)$, where $b=b(s)$ is defined by \eqref{bs}. 
By a reflection with respect to $x=0$, we obtain a solution $w$ of \eqref{w-eq}-\eqref{wbc}.
We conclude that there is a one-to-one correspondence between $\{(s,\ld(s))\mid s\in(1,\infty)\}$ and solutions of \eqref{w-eq}-\eqref{wbc}.

\begin{remark}
It is easy to check from \eqref{bs} that the function $b(s)$ is not monotone in $s$. See also Figure~\ref{Fig1}.
\end{remark}

\section{Proof of Theorem~\ref{th:main}}
\setcounter{equation}{0}

This section is devoted to the proof of our main theorem, Theorem~\ref{th:main}.
Note that \eqref{n-lds} is valid for all $\alpha\ge 0$.
Note that $\ld(s)>0$ for all $s\in(1,\infty)$.
Also, it is easy to check from \eqref{n-lds} that $\ld(1^+)=\ld(+\infty)=0$.


Given a fixed $\alpha\ge 0$. Writing \eqref{n-lds} as $\ld=\ld(s)={[C(s)]^2}/\{2[D(s)]^3\}$, where
\beaa
&&C(s):=\sqrt{s(s-1)}+A(s)+4\alpha\left(2s-1+\sqrt{\frac{s-1}{s}}A(s)\right)A(s),\\
&&D(s):=2s-1+\sqrt{\frac{s-1}{s}}A(s).
\eeaa
Then $\ld'(s)=0$ if and only if $2C'(s)D(s)=3C(s)D'(s)$.
To proceed further, we first compute
\beaa
A'(s)&=&\frac{1}{2\sqrt{s(s-1)}},\quad D'(s)=2+\frac{1}{2s}+\frac{A(s)}{2s\sqrt{s(s-1)}},\\
C'(s)&=&\frac{s}{\sqrt{s(s-1)}}+4\alpha\left[2+\frac{1}{2s}+\frac{A(s)}{2s\sqrt{s(s-1)}}\right]A(s)\\
&&\; +2\alpha\left[2s-1+\sqrt{\frac{s-1}{s}}A(s)\right]\frac{1}{\sqrt{s(s-1)}}.
\eeaa
{Hence $\ld'(s)=0$ if and only if $P_\alpha(s):=E(s)+\alpha F(s)=0$, where
\beaa
&&E(s):=\frac{3}{2s\sqrt{s(s-1)}}A^2(s)+\left(4+\frac{3}{s}\right)A(s)+\frac{4s^2-5s-3}{2\sqrt{s(s-1)}},\\
&&F(s):=\frac{2}{s^2}A^3(s)+\frac{2(4s-3)}{\sqrt{s(s-1)}}A^2(s)+\frac{2(2s-1)(4s-3)}{s}A(s)-\frac{4(2s-1)^2}{\sqrt{s(s-1)}}.
\eeaa

\begin{lemma}\label{inc}
Both $E(s)$ and $F(s)$ are strictly increasing for $s>1$.
\end{lemma}
\pf
First, we compute
\beaa
E'(s)&=&\frac{8s^3-4s^2+9s-9}{4s(s-1)\sqrt{s(s-1)}}-\frac{6s-9}{2s^2(s-1)}A-\frac{3(4s-3)}{4s^2(s-1)\sqrt{s(s-1)}}A^2\\
&:=&\frac{1}{4s^2(s-1)\sqrt{s(s-1)}}G(s), \mbox{ where}
\eeaa
\beaa
G(s):=(8s^4-4s^3+9s^2-9s)-2(6s-9)\sqrt{s(s-1)}A-3(4s-3)A^2.
\eeaa
Using the property that $A(s)<\sqrt{s-1}<\sqrt{s(s-1)}$ for $s>1$, we have
\beaa
G(s)&>&(8s^4-4s^3+9s^2-9s)+{6\sqrt{s(s-1)}A-2( 6s-6)[s(s-1)]}-3(4s-3)(s-1)\\
 &>&8s^4-16s^3+21s^2-9=8s^2(s-1)^2+(13s^2-9)>0,\; \forall s>1.
\eeaa
{Hence $E'(s)>0$ for $s>1$.}

Next, we consider the function $F(s)$. As before, we first compute
\beaa
F'(s)&=&-\frac{4}{s^{3}}A^3+\frac{2s^3-3}{s^2(s-1)\sqrt{s(s-1)}}A^2\\
&&\quad +\frac{2(2s-1)(4s^2-3)}{s^2(s-1)}A-\frac{(4s-5)(2s-1)(s+1)}{s(s-1)\sqrt{s(s-1)}}.
\eeaa
We rewrite $F'(s)=G_1(s)+G_2(s)$, where
\beaa
&&G_1(s):=-\frac{4}{s^{3}}A^3-\frac{1}{s^2(s-1)\sqrt{s(s-1)}}A^2+\frac{2(2s-1)}{s^2(s-1)}A,\\
&&G_2(s):=\frac{2s^3-2}{s^2(s-1)\sqrt{s(s-1)}}A^2+\frac{2(2s-1)(4s^2-4)}{s^2(s-1)}A-\frac{(4s-5)(2s-1)(s+1)}{s(s-1)\sqrt{s(s-1)}}.
\eeaa
We claim that both $G_1(s)$ and $G_2(s)$ are positive for $s>1$.

For $G_1$, we use again $A(s)<\sqrt{s-1}<\sqrt{s(s-1)}$ and obtain
\beaa
G_1(s)>\frac{A}{s^3(s-1)}\{2s(2s-1)-4(s-1)^2-s\}=\frac{A}{s^3(s-1)}(5s-4)>0
\eeaa
for $s>1$.
For $G_2$, we rewrite it as
\beaa
G_2(s)=\frac{2s^3-2}{s^2(s-1)\sqrt{s(s-1)}}A^2+\frac{8(2s-1)(s+1)}{s^2}g(s),\;g(s):=A-\frac{s(4s-5)}{8(s-1)\sqrt{s(s-1)}}.
\eeaa
We compute
\beaa
g'(s)=\frac{8s^2-14s+3}{16(s-1)^2\sqrt{s(s-1)}}.
\eeaa
It is easy to see that $g(s)$ has a global minimum at $s=3/2$ for $s\in(1,\infty)$. Since $g(3/2)>0$, so $g(s)>0$ for $s>1$.
Hence $G_2(s)>0$ for $s>1$. We conclude that $F'(s)>0$ for $s>1$ and the lemma is proved.}
\qed

{Now, we are ready to prove Theorem~\ref{th:main} as follows.

\noindent{\bf Proof of Theorem~\ref{th:main}.} Given a fixed $\alpha\ge 0$. First, we note that $P_\alpha(1^+)=-\infty$ and $P_\alpha(\infty)=\infty$.
Since $P_\alpha(s)$ is strictly increasing in $s$, by Lemma~\ref{inc}, there is a unique $s^*\in(1,\infty)$ (depending on $\alpha$) such that $P_\alpha(s^*)=0$.
Hence $\ld'(s)=0$ if and only if $s=s^*$.
Therefore, Theorem~\ref{th:main} follows immediately with $\ld^*:=\ld(s^*)$, since $\ld(s)>0$ for $s>1$ and $\ld(1^+)=\ld(+\infty)=0$.}
\qed

\section{Acknowledgements}
This work was partially supported by the Ministry of Science and Technology of Taiwan under the grants 108-2115-M-032-006-MY3 (JSG) and 107-2115-M-194-002-MY2 (CJW).



\end{document}